\documentclass[11pt]{amsart}
\usepackage[active]{srcltx}
\usepackage{calc,amssymb,amsthm,amsmath,amscd, ulem}
\usepackage{alltt}
\usepackage[left=1.35in,top=1.25in,right=1.35in,bottom=1.25in]{geometry}
\RequirePackage[dvipsnames,usenames]{color}

\DeclareFontFamily{OMS}{rsfs}{\skewchar\font'60}
\DeclareFontShape{OMS}{rsfs}{m}{n}{<-5>rsfs5 <5-7>rsfs7 <7->rsfs10 }{}
\DeclareSymbolFont{rsfs}{OMS}{rsfs}{m}{n}
\DeclareSymbolFontAlphabet{\scr}{rsfs}

\newtheorem{theorem}{Theorem}[section]

\newtheorem{proposition}[theorem]{Proposition}
\newtheorem{corollary}[theorem]{Corollary}

\theoremstyle{definition}
\newtheorem{definition}[theorem]{Definition}

\theoremstyle{remark}
\newtheorem{remark}[theorem]{Remark}

\newtheorem{question}[theorem]{Question}
\newtheorem{problem}[theorem]{Problem}

\numberwithin{equation}{theorem}

\newcommand{\bG}{\mathbb{G}}

\newcommand{\bZ}{\mathbb{Z}}

\DeclareMathOperator{\Spec}{{Spec}}

\DeclareMathOperator{\Supp}{{Supp}}

\DeclareMathOperator{\Hom}{Hom}

\DeclareMathOperator{\Ext}{Ext}

\DeclareMathOperator{\mdim}{{mdim}}
\DeclareMathOperator{\embdim}{{embdim}}

\DeclareMathOperator{\hgt}{{ht}}
\DeclareMathOperator{\lcd}{{lcd}}
\DeclareMathOperator{\Fdepth}{{{\it F}-depth}}

\newcommand{\bm}{\mathfrak{m}}
\newcommand{\bn}{\mathfrak{n}}
\newcommand{\ba}{\mathfrak{a}}
\newcommand{\bb}{\mathfrak{b}}
\newcommand{\fp}{\mathfrak{p}}
\newcommand{\fq}{\mathfrak{q}}

\input{xy}
\xyoption{all}
\usepackage{tikz}

\newcommand{\bp}{\mathfrak{p}}
\newcommand{\bq}{\mathfrak{q}}

\usepackage{fullpage}

\usepackage{setspace}
\usepackage{hyperref}

\usepackage{enumerate}

\usepackage{graphicx}

\usepackage[all,cmtip]{xy}

\usepackage{verbatim}

\makeatletter
\@namedef{subjclassname@2020}{%
  \textup{2020} Mathematics Subject Classification}
\makeatother

 \title{The Second Vanishing Theorem for Local Cohomology Modules}
\author{Wenliang Zhang}\address{Department of Mathematics, Statistics, and Computer Science, University of Illinois at Chicago, Chicago, IL 60607}
\email{wlzhang@uic.edu}
\thanks{The author is partially supported by NSF through DMS-1752081.}

\begin{document}
\maketitle

\begin{abstract}
We prove the Second Vanishing Theorem for local cohomology modules of an unramified noetherian regular local ring in its full generality. As an application of our vanishing theorem for unramified regular local rings, we extend our topological characterization of the highest Lyubeznik number of an equal-characteristic local ring to the setting of mixed characteristic. Some observations and open questions are also presented along the way.
\end{abstract}

\section{Introduction}

The study of vanishing of local cohomology modules has a long and rich history. In \cite[p.~79]{HartshorneLocalCohomology}, Grothendieck stated the following problem.

\begin{problem}[Grothendieck]
\label{Grothendieck problem}
Let $R$ be a commutative noetherian local ring, $\ba$ be an ideal of $R$, and $n$ be an integer. Find conditions under which $H^i_{\ba}(M)=0$ for all $i>n$ and all $R$-modules $M$.
\end{problem}

Grothendieck proved that $H^i_{\ba}(M)=0$ for all $i>\dim(R)$ and all $R$-modules $M$ (\cite{HartshorneLocalCohomology}), which solved Problem \ref{Grothendieck problem} for $n=\dim(R)$. A solution to Problem \ref{Grothendieck problem} in the case when $n=\dim(R)-1$ was found in \cite[Theorem 3.1]{HartshorneCohomologicalDimension} and has been referred as the Hartshorne-Lichtenbaum Vanishing Theorem. To explain a solution to Problem \ref{Grothendieck problem} in the case when $n=\dim(R)-2$, we consider the following definition:

\begin{definition}
Let $(R,\bm)$ be a $d$-dimensional noetherian local ring and let $\widetilde{R}$ denote the completion of the strict henselization of the completion of $R$. We say that {\it the Second Vanishing Theorem holds for $R$} if, for each ideal $\ba$ in $R$, the following conditions are equivalent:
\begin{enumerate}
\item $H^j_{\ba}(M)=0$ for all $j>d-2$ and all $R$-modules $M$;
\item $\dim(R/\ba)\geq 2$ and the punctured spectrum of $\widetilde{R}/\ba\widetilde{R}$ is connected.
\end{enumerate}
\end{definition}

If $R$ is not regular, then the Second Vanishing Theorem may not hold for $R$, {\it cf.} \cite[7.7]{HochsterZhangContent}. When $R$ is regular, some positive results are known. When $R$ is a polynomial ring over a field and $\ba$ is a homogeneous ideal, then the Second Vanishing Theorem holds, as proved by Hartshorne (\cite[7.5]{HartshorneCohomologicalDimension}) who also coined the name `Second Vanishing Theorem' and proposed the following problem in \cite[p.~445]{HartshorneCohomologicalDimension}:
\begin{problem}[Hartshorne]
\label{Hartshorne Problem}
Prove that the Second Vanishing Theorem holds for all regular local rings.
\end{problem}

It is clear that a solution to Problem \ref{Hartshorne Problem} produces a solution to Grothendieck's original Problem \ref{Grothendieck problem} for noetherian regular local rings when $n=\dim(R)-2$. Subsequently, Ogus (\cite[Corollary~2.11]{OgusLocalCohomologicalDimension}) proved that the Second Vanishing Theorem holds for regular local rings of equal-characteristic 0 and Peskine-Szpiro showed in \cite[III~5.5]{PeskineSzpiroDimensionProjective} that the Second Vanishing Theorem holds for regular local rings of equal-characteristic $p$. \cite{HunekeLyubeznikVanishing} provided a unified proof that the Second Vanishing Theorem holds for regular local rings of equal-characteristic. Extending the Second Vanishing Theorem to rings that do not contain a field has been a major open problem in the study of local cohomology. In \S\ref{section: vanishing unramified}, we resolve Problem \ref{Hartshorne Problem} for unramified regular local rings of mixed characteristic as follows: 

\begin{theorem}
\label{vanishing unramified}
If $R$ is an unramified noetherian regular local ring of mixed characteristic, then the Second Vanishing Theorem holds for $R$. 
\end{theorem}

A special case of Theorem \ref{vanishing unramified}, when $\dim(R/\ba)\geq 3$ {\it and} $R/\ba$ is equidimensional, can be found in \cite{HNBPW_JA_2018}.

In prime characteristic $p>0$, we produce a new proof of the Second Vanishing Theorem using the action of Frobenius and some equivalent formulations of the Second Vanishing Theorem in characteristic $p>0$, which can be found in \S\ref{char p}.

As an application of Theorem \ref{vanishing unramified}, we extend our results in \cite{ZhangHighestLyubeznikNumbers} to local rings of mixed characteristic. Before stating our extension, we recall the definition of the Hochster-Huneke graph of a local ring. Let $A$ be a noetherian local ring. The \emph{Hochster-Huneke graph $\Gamma_A$ of $A$} is defined as follows. Its vertices are the top-dimensional minimal prime ideals of $A$, and two distinct
 vertices $P$ and $Q$ are joined by an edge if and only if $ht_B(P+Q)=1$. 

The main result in \S\ref{L number} is the following:
\begin{theorem}
\label{L numbers mixed char}
Let $(A,\bm,k)$ be a $d$-dimensional noetherian local ring. Assume that $A=R/I$ where $R$ is an $n$-dimensional unramified regular local ring $(R,\bm)$ of mixed characteristic $(0,p)$. Then $\dim_k(\Hom_R(k,H^d_{\bm}H^{n-d}_I(R)))$ is the number of connected components of the Hochster-Huneke graph $\Gamma_{\widetilde{A}}$ of $\widetilde{A}$, where $\widetilde{A}$ is the completion of the strict henselization of the completion of $A$.
\end{theorem}

All rings considered in this article are noetherian.

The paper is organized as follows: In \S\ref{section: vanishing unramified}, we prove Theorem \ref{vanishing unramified}. In \S\ref{section: reduction dim 2}, we reduce the Second Vanishing Theorem to prime ideals of dimension 2. In \S\ref{char p}, we consider some equivalent characterizations of the Second Vanishing Theorem in prime characteristic $p$. In \S\ref{L number}, we prove Theorem \ref{L numbers mixed char}. In our last section \S\ref{observation and question}, we present some observations along with some open questions.
\section{Proof of Theorem \ref{vanishing unramified}}
\label{section: vanishing unramified}

To prove Theorem \ref{vanishing unramified}, we need the following result from \cite{PeskineSzpiroDimensionProjective}.

\begin{theorem}[Th\'eor\`eme~III.5.1 in \cite{PeskineSzpiroDimensionProjective}]
\label{PS theorem}
Let $(R,\bm)$ be a $d$-dimensional complete regular local ring with a separably closed residue field and let $\ba$ be an ideal of $R$. Assume that $\Spec(R/\ba)\backslash\{\bm\}$ is connected and $\dim(R/\ba)\geq 2$. Then the following are equivalent.
\begin{enumerate}
\item $H^i_{\ba}(R)$ is artinian for all $i\geq d-1$;
\item $H^i_{\ba}(R)=0$  for all $i\geq d-1$.
\end{enumerate}
\end{theorem}

\begin{remark}
\label{cofiniteness}
There are many characterizations of artinianness; one such characterization ({\it cf.} \cite[Remark~1.3]{HunekeKohCofiniteness}) asserts that, if $(A,\bm,k)$ is complete local ring, then an $A$-module $M$ is artinian if and only if that $\Supp(M)=\{\bm\}$ and $M$ has finite dimensional socle, {\it i.e.} $\Hom_R(k,M)$ is a finite dimensional $k$-space.
\end{remark}

We now prove Theorem \ref{vanishing unramified}.
\begin{proof}[Proof of Theorem \ref{vanishing unramified}]
Since $H^j_{\ba}(M)\otimes_R\widetilde{R}\cong H^j_{\ba}(M\otimes_R\widetilde{R})$ by flat base change and $\widetilde{R}$ is faithfully flat over $R$, we may assume that $R$ is complete with a separably closed residue field by replacing $R$ with $\widetilde{R}$. 

(1)$\Rightarrow$(2). If $\dim(R/\ba)\leq 1$, then either $H^{d-1}_{\ba}(R)\neq 0$ (when $\dim(R/\ba)=1$) or $H^d_{\ba}(R)\neq 0$ (when $\dim(R/\ba)=0$). Hence $\dim(R/\ba)\geq 2$. If the punctured spectrum of $R/\ba$ were disconnected, then there would be two ideals $I,J$ of height at most $d-1$ such that $I\cap J=\ba$ and $\sqrt{I+J}=\bm$. The Mayer-Vietoris sequence says
\[0=H^{d-1}_{\ba}(R)\to H^d_{I+J}(R)=H^d_{\bm}(R)\to H^d_I(R)\oplus H^d_J(R).\]
Since $ht(I),ht(J)\leq d-1$, the Hartshorne-Lichtenbaum Vanishing Theorem (\cite[Theorem 3.1]{HartshorneCohomologicalDimension}) implies $H^d_i(R)=H^d_J(R)=0$ which would imply that $H^d_{\bm}(R)=0$, a contradiction. Hence the punctured spectrum of $R/\ba$ must be connected.

$(2)\Rightarrow(1)$. It was observed in \cite{HartshorneCohomologicalDimension} that $H^j_I(M)=0$ for all $j>t$ and all $R$-modules $M$ if and only if $H^j_I(R)=0$ for all $j>t$. Hence it suffices to show $H^{d-1}_{\ba}(R)=H^d_{\ba}(R)=0$ (since $H^{>d}_{\ba}(R)=0$ by Grothendieck Vanishing). Combining Theorem \ref{PS theorem} and Remark \ref{cofiniteness}, it suffices to show that both of $H^d_{\ba}(R)$ and $H^{d-1}_{\ba}(R)$ are supported only at the maximal ideal and have finite dimensional socle. It follows from the Hartshorne-Lichtenbaum vanishing theorem (\cite[Theorem 3.1]{HartshorneCohomologicalDimension}) that $H^d_{\ba}(R)=0$ and $\Supp(H^{d-1}_{\ba}(R))=\{\bm\}$. On the other hand, \cite[Theorem 1]{LyubeznikUnramifiedRegular} (or \cite[Theorem 1.2]{NunezBetancourtIJM2013}) shows that $H^{d-1}_{\ba}(R)$ has finite dimensional socle since $R$ is an unramified complete regular local ring. This finishes the proof.
\end{proof}

\begin{remark} 
\label{rmk: ramified case open}
The analogue of Theorem \ref{vanishing unramified} in the {\it ramified} case remains open in general. 

Our approach in the proof of Theorem \ref{vanishing unramified} provides a unified approach to the Second Vanishing Theorem: the same proof also works for equal-characteristic regular local rings. It follows from our proof of Theorem \ref{vanishing unramified} that: let $(R,\bm,k)$ be an $n$-dimensional {\it ramified} complete regular local ring of mixed characteristic $(0,p)$ and $\ba$ be an ideal such that $\Spec(R/\ba)\backslash\{\bm\}$ is connected and $\dim(R/\ba)\geq 2$, if $\dim_{k}\Hom_R(k,H^{n-1}_{\ba}(R))<\infty$, then $H^{n-1}_{\ba}(R)=0$. 

Note that $\dim_{k}\Hom_R(k,H^{n-1}_{\ba}(R))$ is one of the Bass numbers of the local cohomology module $H^{n-1}_{\ba}(R)$. One ought to remark that the finiteness of Bass numbers of local cohomology modules of a {\it ramified} regular local ring of mixed characteristic was first conjectured in \cite{HunekeProblemsLC} and has been a long standing open problem. This is one of the reasons we consider a reduction and some characterizations of the Second Vanishing Theorem in characteristic $p$ in the subsequent sections.

\S\ref{observation and question} contains some open questions which may be viewed as different approaches to the Second Vanishing Theorem in mixed characteristic.
\end{remark}

\section{Reduction to dimension 2}
\label{section: reduction dim 2}
In this section, we show that the Second Vanishing Theorem can be reduced to the case when the ideal $\ba$ is a prime ideal of dimension 2, {\it i.e.} $\dim(R/\ba)=2$, and discuss a related approach to proving the Second Vanishing Theorem in general.  

We begin with the following result which was implicitly contained in the proof of \cite[Theorem~2.9]{HunekeLyubeznikVanishing}. 
\begin{proposition}
\label{connected subgraph}
Let $(A,\bm)$ be a catenary noetherian local ring and $\ba$ be an ideal of $A$. Let $\{\bp_1,\dots,\bp_t\}$ be the set of minimal primes of $\ba$. Assume that $\Spec(A/\ba)\backslash\{\bm\}$ is connected. Then there exists $\bp_i$ such that $\Spec(A/\bb_i)\backslash\{\bm\}$ is connected where $\bb_i=\bigcap_{j\neq i}\bp_j$.
\end{proposition}
\begin{proof}
The proof follows the same line of reasoning as in line 7-15 from the bottom on page 79 in \cite{HunekeLyubeznikVanishing}. Since \cite[Theorem~2.9]{HunekeLyubeznikVanishing} only treats rings of equal-characteristic, we opt to reproduce the proof here for the sake of completeness and clarity.

When $t=1$, there is nothing to prove. Assume that $t\geq 2$. Then $\dim(A/\bp_j)\geq 2$ for each $1\leq j\leq t$ (otherwise $\bp_j+\bigcap_{i\neq j}\bp_i$ would be $\bm$-primary and hence $\Spec(A/\ba)\backslash\{\bm\}$ would be disconnected). Consider the graph $\bG$ with vertices $1,\dots,t$ in which $i$ and $j$ are joined by an edge if $\bp_i+\bp_j$ is {\it not} $\bm$-primary. One can see that $\Spec(A/\ba)\backslash\{\bm\}$ is connected if and only if so is this graph $\bG$. Since $\bG$ is a finite connected graph, it admits a spanning tree by \cite[Corollary 5, p.~7]{BollobasBookGraphTheory}. Since the number of edges in a tree is less than the number of vertices, it follows from \cite[Corollary 7, p.~8]{BollobasBookGraphTheory} that this tree has a vertex from which only one edge emanates. The graph $\bG$ minus this particular vertex (and its only edge) is connected. Assume this particular vertex is $i$. Then, $\Spec(A/\bb_i)\backslash\{\bm\}$ is connected where $\bb_i=\bigcap_{j\neq i}\bp_j$.
\end{proof}

Our main result is the following:
\begin{theorem}
\label{prop: reduce to dim 2}
Let $(R,\bm,k)$ be a $d$-dimensional complete regular local ring with a separably closed residue field. If $H^d_{\fp}(R)=H^{d-1}_{\fp}(R)=0$ for all prime ideals $\fp$ with height $d-2$, then $H^d_{\ba}(R)=H^{d-1}_{\ba}(R)=0$ for all ideals $I$ such that $\Spec(R/\ba)\backslash \{\bm\}$ is connected and that $\dim(R/\ba)\geq 2$.
\end{theorem}
\begin{proof}
We will use reverse induction on the height of $\ba$. Since the conclusion only depends on the radical of $\ba$, we may assume that $\ba$ is radical. First, assume $\hgt(\ba)=d-2$ and write $\ba=\cap^t_{i=1} \fp_i$. Since $\Spec(R/\ba)\backslash \{\bm\}$ is connected, each prime $\fp_i$ must have height $d-2$. The Hartshorne-Lichtenbaum Vanishing Theorem implies that $H^d_{\ba}(R)=0$ We will use induction on $t$ to show that $H^{d-1}_{\ba}(R)=0$. When $t=1$, this is precisely our assumption. Assume that $t\geq 2$. By Proposition \ref{connected subgraph} and renumbering the minimal primes if necessary, we may assume that $\Spec(R/\cap_{j\geq 2}\fp_j)\backslash\{\bm\}$ is connected. We consider the exact sequence 
\[\cdots \to H^{d-1}_{\fp_1}(R)\oplus H^{d-1}_{\cap ^t_{i=2}\fp_i}(R) \to H^{d-1}_{\ba}(R)\to H^d_{\fp_1+\cap ^t_{i=2}\fp_i}(R).\]
Since $\Spec(R/\ba)\backslash \{\bm\}$ is connected, $\fp_1+\cap ^t_{i=2}\fp_i$ is not $\bm$-primary; consequently it follows from Hartshorne-Lichtenbaum Vanishing Theorem that $H^d_{\fp_1+\cap ^t_{i=2}\fp_i}(R)=0$. The induction hypothesis asserts that $H^{d-1}_{\cap ^t_{i=2}\fp_i}(R)=0$. Therefore, $H^{d-1}_{\ba}(R)=0$.

Now assume that $\hgt(\ba)\leq d-3$ and write $\ba=\fp_1\cap\cdots\cap \fp_t$. We prove our result by induction on $t$. First, we treat the case when $t=1$ (that is, when $\ba$ is a prime ideal). In this case pick $r\in \bm\backslash \ba$, then by Faltings' connectedness theorem $\Spec(R/(\ba,r))\backslash \{\bm\}$ is connected since the number of generators of $(r)$ in $R/\ba$ is at most $\dim(R/\ba)-2$. As $\hgt((\ba,r))=\hgt(\ba)+1$, our induction hypothesis implies that $H^{d-1}_{(\ba,r)}(R)=H^{d}_{(\ba,r)}(R)=0$. Since $\Supp(H^{d-1}_{\ba}(R))\subseteq \{\bm\}$, we have $H^{d-1}_{\ba}(R)_r=0$. Now the exact sequence 
\[\cdots\to H^{j-1}_{(\ba,r)}(R) \to H^j_{\ba}(R) \to H^j_{\ba}(R)_r\to\cdots\] 
shows that $H^d_{\ba}(R)=H^{d-1}_{\ba}(R)=0$. Next, we assume that $t\geq 2$. Since $\Spec(A/\ba)\backslash\{\bm\}$ is connected, it follows from Proposition \ref{connected subgraph} that there exists $\bp_i$ such that $\Spec(A/(\cap_{j\neq i}\bp_j))\backslash\{\bm\}$ is connected. By rearranging the indices, we may assume that $i=t$. Set $I=\bp_1\cap\cdots\cap \bp_{t-1}$ and $J=\bp_t$. Then one can check that $\sqrt{I+J}\subsetneq \bm$ (since $\Spec(R/\ba)\backslash \{\bm\}$ is connected) and $I\cap J=\ba$. Then the induction hypothesis implies that $H^{d-1}_I(R)=0$; the case when $t=1$ implies that $H^{d-1}_J(R)=0$. Moreover, the Hartshorne-Lichtenbaum Vanishing Theorem implies that $H^d_{I+J}(R)=0$. It follows from the Mayer-Vietoris sequence 
\[0=H^{d-1}_I(R)\oplus H^{d-1}_J(R) \to H^{d-1}_{\ba}(R) \to H^d_{I+J}(R)=0\] 
that $H^d_{\ba}(R)=H^{d-1}_{\ba}(R)=0$. 
\end{proof}

\section{Second Vanishing Theorem in prime characteristic $p$, Revisited}
\label{char p}

In this section we consider the Second Vanishing Theorem in characteristic $p$ from a different perspective, including some equivalent characterizations.

We recall the following result due to Lyubeznik (\cite[Theorem~1.1]{LyubeznikVanishingLCCharp}).
\begin{theorem}
\label{theorem: vanishing of lc in char p}
Let $(R,\bm)$ be a regular local ring of dimension $n$ of prime characteristic $p$ and $A$ be a homomorphic image of $R$. Let $I$ be the kernel of $R\twoheadrightarrow A$. Then $H^{n-i}_I(R)=0$ if and only if the Frobenius action on $H^i_{\bm}(A)$ is nilpotent.
\end{theorem}

\begin{remark}
\label{algebraically closed residue field}
In the statement of the Second Vanishing Theorem, the ring $\widetilde{R}$ is complete and its residue field is separably closed. In prime characteristic $p$ (the setting of this section), $\widetilde{R}$ contains a coefficient field $k$ which is isomorphic to its residue field; consequently $\tilde{R}\cong k[[x_1,\dots,x_n]]$ where $n=\dim(\tilde{R})$. Let $k'$ be an algebraic closure of $k$ and set $R'$ to be $\widetilde{R}\otimes_k k'$. Since $k$ is separably closed, the field extension $k'/k$ is a directed union of finite purely inseparable extensions $k_e$ of $k$ for $e\geq 1$. Hence 
\[R'\cong \varinjlim_e(k[[x_1,\dots,x_n]]\otimes_kk_e)\cong \varinjlim_ek_e[[x_1,\dots,x_n]]\]
where the second isomorphism follows from the fact $k_e/k$ is a finite extension. Consequently, 
\begin{itemize}
\item $R'$ is a noetherian local ring with the maximal ideal $(x_1,\dots,x_n)$ ({\it cf.} \cite[Theorem~1]{OgomaDirectLimit});
\item $R'$ is henselian (as a direct limit of henselian local rings); and
\item the completion $\widehat{R'}$ of $R'$ is isomorphic to $k'[[x_1,\dots,x_n]]$.
\end{itemize}
The ring homomorphism $\widetilde{R}/\ba\widetilde{R}\to R'/\ba R'$ induces a homeomorphism on the spectra ({\it cf.} \cite[\href{https://stacks.math.columbia.edu/tag/0BR5}{Tag 0BR5}]{stacks-project}) since $k_e/k$ is purely inseparable. Therefore, the punctured spectrum of $\widetilde{R}/{\ba}\widetilde{R}$ is connected if and only if the same holds for $R'/\ba R'$. Furthermore, since $R'$ is a noetherian henselian local ring, it follows directly from \cite[\href{https://stacks.math.columbia.edu/tag/0C2C}{Tag 0C2C}]{stacks-project} that the punctured spectrum of $R'/\ba R'$ is connected if and only so is the one of $\widehat{R'}/\ba \widehat{R'}$. It is clear that passing from $\widetilde{R}$ to $\widehat{R'}$ does not affect vanishing of local cohomology. Hence when proving the Second Vanishing Theorem, one may assume that the residue field is algebraically closed.
\end{remark}

\begin{proposition}
\label{Serre implies finite length}
Let $(A,\bm)$ be a noetherian local ring which is a homomorphic image of a Gorenstein local ring. If $A$ satisfies Serre's condition $(S_i)$ for some $i<\dim(A)$, then $H^i_{\bm}(A)$ has finite length.
\end{proposition}
\begin{proof}
Let $(B,\bm)$ be a Gorenstein local ring with a surjection $B\to A$ with kernel $J$. Set $d=\dim(B)$. Local Duality asserts that 
\[H^i_{\bm}(A)\cong \Ext^{d-i}_B(A,B)^{\vee}\]
Hence it suffices to show that $\Ext^{d-i}_B(A,B)$ is supported in $\{\bm\}$. Let $P$ be a prime ideal of height $d-1$ in $B$ and consider $\Ext^{d-i}_{B_P}(A_P,B_P)$. If $P$ does not contain $J$, then $A_P=0$. We may assume that $P$ contains $J$ (and hence induces a prime ideal in $A$). Since the depth of $A_P$ is at least $i$ and $\dim(B_P)=d-1$, we have 
\[\Ext^{d-i}_{B_P}(A_P,B_P)^{\vee}\cong H^{(d-1)-d+i}_{PB_P}(A_P)=H^{i-1}_{PB_P}(A_P)=0.\] 
This finishes the proof.
\end{proof}

We are in position to present a proof of the Second Vanishing Theorem in characteristic $p$, using an approach different from the one in \cite{PeskineSzpiroDimensionProjective}.
\begin{proof}[An alternative proof of Second Vanishing Theorem in characteristic $p$]
Let $(R,\bm)$ be a noetherian regular local ring of prime characteristic $p$. Set $n=\dim(R)$ and $I$ be an ideal of height at most $n-2$. Combining Remark \ref{algebraically closed residue field} and Proposition \ref{prop: reduce to dim 2}, we may assume that $R$ is a complete noetherian regular local ring of prime characteristic $p$ with an algebraically closed residual field and $I$ is a height-$(d-2)$ prime ideal. 

The proof of the implication that the vanishing $H^n_I(R)=H^{n-1}_I(R)=0$ implies the connectedness of $\Spec(R/I)\backslash \{\bm\}$ is the same as in the proof of Theorem \ref{vanishing unramified} (the proof of this particular implication is characteristic-free). We will focus on the other implication.

Assume that $\Spec(R/I)\backslash \{\bm\}$ is connected and we wish to show $H^n_I(R)=H^{n-1}_I(R)=0$. The vanishing $H^n_I(R)=0$ follows from Hartshorne-Lichtenbaum vanishing. It remains to show $H^{n-1}_I(R)=0$. Set $A=R/I$. According to Theorem \ref{theorem: vanishing of lc in char p}, this is equivalent to the niloptence of the Frobenius action on $H^1_{\bm}(A)$. 

Since $A$ is a local integral domain and hence satisfies Serre's condition $(S_1)$, by Proposition \ref{Serre implies finite length}, $H^1_{\bm}(A)$ has finite length. Let $f$ denote the Frobenius action on $H^1_{\bm}(A)$. Since $H^1_{\bm}(A)$ has finite length, every element in $\bm H^1_{\bm}(A)$ is $f$-nilpotent. It follows that $H^1_{\bm}(A)$ is $f$-nilpotent if and only if $(H^1_{\bm}(A))_s:=\bigcap_{t}f^t(H^1_{\bm}(A))$ is 0. 

Set $U=\Spec(R/I)\backslash \{\bm\}$. Then \cite[3.1]{HartshorneSpeiserLocalCohomologyInCharacteristicP} asserts that there is an exact sequence
\[0\to H^0_{\bm}(A)_s\to A_s\to H^0(U,\mathcal{O}_U)_s\to H^1_{\bm}(A)_s\to 0.\]
Since $U$ is connected and $A$ is a domain (whose residue field is algebraically closed), the map in the middle $A_s\to H^0(U,\mathcal{O}_U)_s$ is the isomorphism $k\xrightarrow{\sim} k$ where $k$ is the residue field of $A$. Hence $H^1_{\bm}(A)_s=0$. This proves that $H^1_{\bm}(A)$ is $F$-nilpotent and hence $H^{n-1}_I(R)=0$. 
\end{proof}

Next we will consider some equivalent formulations of the Second Vanishing Theorems in characteristic $p$. To this end, we recall some basic facts regarding $S_2$-ification from \cite{HochsterHunekeIndecomposable}. 

\begin{remark}
Let $(A,\bm)$ be a complete local domain with a canonical module $\omega$, then 
\begin{enumerate}
\item $\Hom_A(\omega,\omega)$ is a commutative complete local ring and the natural map $A\to \Hom_A(\omega,\omega)$ is an injective module-finite ring homomorphism;
\item $\Hom_A(\omega,\omega)$ satisfies Serre's $(S_2)$-condition as both an $A$-module and as a ring on its own. 
\end{enumerate}
\end{remark}

\begin{theorem}
\label{vanishing S2}
The following statements are equivalent:
\begin{enumerate}
\item The Second Vanishing Theorem holds for regular local rings of characteristic $p$.
\item Let $(A,\bm)$ be a $2$-dimensional complete local domain of prime characteristic $p>0$ with a canonical module $\omega$. Assume $A/\bm$ is separably closed. For each element $\phi\in \Hom_A(\omega,\omega)$ there is an integer $e$ such that $\phi^{p^e}\in A$, {\it i.e.} there is an element $a\in A$ such that $\phi^{p^e}$ is the multiplication by $a$ on $\omega$.
\item Let $A$ be a 2-dimensional complete local domain of characteristic $p$ with a separably closed residue field. Then there exists a positive integer $n$ such that, for all systems of parameters $x,y$, one has $(x:y)^{[p^{n'}]}=(x^{p^{n'}})$ for all $n'\geq n$. 
\end{enumerate} 
\end{theorem}

\begin{proof}
First we prove that $(1)\Rightarrow (2)$. Assume that the Second Vanishing Theorem holds in characteristic $p$. Let $(A,\bm)$ be as in $(2)$. Set $S:=\Hom_A(\omega,\omega)$. Consider the short exact sequence $0\to A\to S\to C\to 0$. (2) is equivalent to proving that the Frobenius on $C$ is nilpotent. 

The short exact sequence induces a long exact sequence on local cohomology 
\[0=H^0_{\bm}(S)\to H^0_{\bm}(C)\to H^1_{\bm}(A)\to H^1_{\bm}(S)=0\]
where $H^0_{\bm}(S)=H^1_{\bm}(S)=0$ since $S$ satisfies $(S_2)$-condition. Write $A=R/I$ where $R$ is an $n$-dimensional complete regular local ring. By the Second Vanishing Theorem, $H^{n-1}_I(R)=0$. By Theorem \ref{theorem: vanishing of lc in char p}, the Frobenius on $H^1_{\bm}(A)$ must be nilpotent. Hence so is the Frobenius on $H^0_{\bm}(C)$. We claim that $H^0_{\bm}(C)=C$ and we reason as follows.  It suffices to show that $C$ is supported in the maximal ideal only. Let $\fp$ be any non-maximal prime ideal. Since $A_{\fp}$ is Cohen-Macaulay (note that $\dim(A)=2$), we have $S_{\fp}\cong A_{\fp}$ and hence $C_{\fp}=0$. This shows that $C$ is supported in the maximal ideal only, thus $H^0_{\bm}(C)=C$. So, the Frobenius on $C$ is nilpotent and, equivalently, $(2)$ holds.

Next, we prove $(2)\Rightarrow (1)$. Assume now $(2)$ holds, and we wish to prove the Second Vanishing Theorem. To this end, let $(R,\bm)$ be an $n$-dimensional prime-characteristic complete regular local ring with a separably closed residue field. By Proposition \ref{prop: reduce to dim 2}, it suffices to prove that $H^{n-1}_P(R)=0$ for all prime ideals $P$ of height $n-2$. Set $A=R/P$. Then $A$ satisfies that hypotheses in $(2)$. By the argument in the previous paragraph, we see that the Frobenius on $H^1_{\bm}(A)$ is nilpotent. According to Theorem \ref{theorem: vanishing of lc in char p}, we have $H^{n-1}_P(R)=0$. This completes the proof of $(2)\Rightarrow (1)$ and hence $(1)\Leftrightarrow(2)$.

Next we prove that $(1)\Leftrightarrow (3)$. We have seen that $(1)$ is equivalent to $H^1_{\bm}(A)$ being $f$-nilpotent where $(A,\bm)$ is a 2-dimensional complete local domain of characteristic $p$ with a separably closed residue field. Given an arbitrary system of parameters $x,y$ in $A$, each element in $H^1_{\bm}(A)$ can be written as $[\frac{a}{x},\frac{b}{y}]$ such that $ay=bx$. Since $H^1_{\bm}(A)$ is artinian, it is $f$-nilpotent if and only if there is an integer $n$ such that $f^{n}(H^1_{\bm}(A))=0$ (and consequently $f^{n'}(H^1_{\bm}(A))=0$ for all $n'\geq n$). This holds if and only if $[\frac{a}{x},\frac{b}{y}]^{p^{n'}}=0$ for all $a,b,x,y$ such that $ay=bx$. Note that $[\frac{a}{x},\frac{b}{y}]^{p^{n'}}=0$ if and only if $a^{p^{n'}}\in (x^{p^{n'}})$ and $ay=bx$ if and only if $a\in (x:y)$. This completes the proof of $(1)\Leftrightarrow (3)$.
\end{proof}

\section{The highest Lyubeznik number of a local ring of mixed characteristic}
\label{L number}
In \cite[\S4]{LyubeznikFinitenessLC}, Lyubeznik introduced a set of integers attached to an equal-characteristic complete local ring which have been referred to as Lyubeznik numbers ever since. Because of the connections with topology of algebraic varieties ({\it cf.} \cite{GarciaSabbah, WaltherLyubeznikNumbers, BlickleBondu, ZhangHighestLyubeznikNumbers, ZhangLyubeznikNumbersProjectiveSchemes, SwitalaBLMS, ReicheltSaitoWalther}), studying Lyubeznik numbers bas become an active research area. In \cite{ZhangHighestLyubeznikNumbers}, the author proved a topological characterization of the highest Lyubeznik number for all local rings of equal-characteristic. The goal of this section is to extend the main theorem in \cite{ZhangHighestLyubeznikNumbers} to mixed characteristic, using our Theorem \ref{vanishing unramified}.

Let $(A,\bm,k)$ be a complete local ring of mixed characteristic. By Cohen's structure theorem, $A$ admits a surjection $\pi:R\twoheadrightarrow A$ from a complete unramified regular local ring $(R,\bm,k)$. Let $I$ be the kernel of $\pi$ and $n$ denote $\dim(R)$. We have the following.

\begin{proposition}
Let $A,R,I,n$ be as above. Then 
\[\dim_k\Hom_R(k,H^i_{\bm}H^{n-j}_I(R))\]
depends only on $A,i,j$, but not on the choices of $R$ or $\pi$.
\end{proposition}
\begin{proof}
The proof follows the same line of ideas as in \cite[4.1]{LyubeznikFinitenessLC} and \cite[3.4]{NunezWittLNMixedChar}, hence we will provide a sketch only. 

Since each complete unramified regular local ring is a formal power series ring over a coefficient ring by Cohen's structure theorem, one can reduce the proof to proving the following: 
\begin{equation}
\label{reduction to one variable}
\dim_k\Hom_R(k,H^i_{\bm}H^{n-j}_I(R))=\dim_k\Hom_{R[[x]]}(k,H^i_{(\bm,x)}H^{n+1-j}_{(I,x)}(R[[x]]))
\end{equation}
where $x$ is an indeterminate over $R$. For each $R$-module $M$, define 
\[G(M):=M\otimes_RH^1_{(x)}(R[[x]]);\] 
$G$ was introduced in \cite[Proof of 4.3]{LyubeznikFinitenessLC} and further studied in \cite[\S 3]{NunezWittGenLN}. Then by \cite[3.10]{NunezWittGenLN}, $G(H^i_{\bm}H^{n-j}_I(R))=H^i_{(\bm,x)}H^{n+1-j}_{(I,x)}(R[[x]])$. Now (\ref{reduction to one variable}) follows from \cite[3.12]{NunezWittGenLN} which asserts that $\Hom_R(k,M)=\Hom_{R[[x]]}(k,G(M))$ for all $R$-modules $M$.
\end{proof}

\begin{remark}
By Cohen's Structure Theorem of complete local rings, if $A$ is a complete local ring of mixed characteristic, then $A$ admits a surjection $\pi:R\twoheadrightarrow \hat{A}$ from a complete unramified regular local ring $R$ (with kernel $I$).

If $A$ (not necessarily complete) itself admits a surjection $R'\twoheadrightarrow A$ from an $n'$-dimensional unramified regular local ring $(R',\bm',k)$ with kernel $I'$, then one can check that
\[\dim_k\Hom_R(k,H^i_{\bm}H^{n-j}_I(R))=\dim_k\Hom_{R'}(k,H^i_{\bm'}H^{n'-j}_{I'}(R')).\] 
\end{remark}

\begin{definition}
\label{definition of lambda}
Let $(A,\bm,k)$ be a noetherian local ring of mixed characteristic and let $\hat{A}$ denote its completion. Let $\pi:R\twoheadrightarrow \hat{A}$ be a surjection from an $n$-dimensional complete unramified regular local ring $(R,\bm,k)$ of mixed characteristic. Define
\[\lambda_{i,j}(A):=\dim_k\Hom_R(k,H^i_{\bm}H^{n-j}_I(R)).\]
\end{definition}

Note that, since $R$ is an unramified regular local ring of mixed characteristic, it follows from \cite{LyubeznikUnramifiedRegular} that $\lambda_{i,j}(A)$ are finite.

\begin{remark}
If $(A,\bm,k)$ is a noetherian local ring containing a field and admits presentation $A=R/I$ where $(R,\bm,k)$ is an $n$-dimensional regular local ring containing the same field, then it follows from \cite[1.4,~4.1]{LyubeznikFinitenessLC}) that 
\begin{equation}
\label{Hom Ext agree}
\dim_k\Hom_R(k,H^i_{\bm}H^{n-j}_I(R))=\dim_k\Ext^i_R(k,H^{n-j}_I(R)).
\end{equation}

However, when $A$ does not contain a field, (\ref{Hom Ext agree}) may no longer hold. 

Let $R=\bZ_2[[x_0,\dots,x_5]]$ and let $I$ be the ideal of $R$ generated by the $10$ monomials\footnote{These monomials come from a triangulation of the real projective plane; the interested reader may find related discussions in \cite{SinghWaltherBockstein}.}
\[
\{x_0x_1x_2, x_0x_1x_3, x_0x_2x_4, x_0x_3x_5, x_0x_4x_5, x_1x_2x_5, x_1x_3x_4, x_1x_4x_5, x_2x_3x_4, x_2x_3x_5\}.
\]
Let $\bm$ denote $(2,x_0,\dots,x_5)$ and set $k=R/\bm,A=R/I$. Then it is proved in \cite[5.5]{DattaSwitalaZhang} that 
\[H^4_I(R)\cong H^6_{(x_0,\dots,x_5)}(R/(2)).\]
Since $H^6_{(x_0,\dots,x_5)}(R/(2))$ admits an injective resolution (as an $R$-module)
\[0\to H^6_{(x_0,\dots,x_5)}(R/(2))\to E_R(R/\bm)\xrightarrow{\cdot 2}E_R(R/\bm)\to 0,\]
one can check that
\[\dim_{k}\Hom_R(k,H^1_{\bm}H^4_I(R))=0\neq 1=\dim_k\Ext^1_R(k,H^4_I(R)).\]

\end{remark}

Next, we focus on $\lambda_{d,d}(A)$ where $d=\dim(A)$ and prove our Theorem \ref{L numbers mixed char}, which extends our results in \cite{ZhangHighestLyubeznikNumbers} to local rings of mixed characteristic.

\begin{proof}[Proof of Theorem \ref{L numbers mixed char}]
Since both completion and strict henselization are faithfully flat, we may assume that both $A$ and $R$ are complete with separably closed residue fields. Assume that $\Gamma_1,\dots,\Gamma_t$ are the connected components of $\Gamma_A$. For $1 \leq j \leq t$, let $I_j$ be the intersection of the minimal primes of $A$ that are vertices of $\Gamma_j$. Similar to the proof of \cite[Proposition 2.1]{LyubeznikSomeLCInvariants}, using the Mayer-Vietoris sequence of local cohomology, one can prove that 
\[H^{n-d}_I(R)=\oplus_{j=1}^tH^{n-d}_{I_j}(R).\]
Hence 
\[\dim_k\Hom_R(k,H^d_{\bm}H^{n-d}_I(R))=\sum_{j=1}^t\dim_k\Hom_R(k,H^d_{\bm}H^{n-d}_{I_j}(R)).\]
We are reduced to proving that $\dim_k\Hom_R(k,H^d_{\bm}H^{n-d}_I(R))=1$ when $\Gamma_A=\Gamma_{R/I}$ is connected and $A$ is equidimensional. The rest of the proof follows the same strategy as in \cite{ZhangHighestLyubeznikNumbers}.

We will use induction on $\dim(A)$. 

First, assume that $\dim(A)=\dim(R/I)=2$. Since $\Gamma_A$ is connected and $A$ is equidimensional, $\Spec(A)\backslash \{\bm\}$ is also connected. Since $R$ is regular (and hence Gorenstein), it admits an injective resolution of the following form
\[0\to R\to \cdots \to \bigoplus_{I\subseteq \bp;\ \hgt(\bp)=j}E(R/\bp)\to \cdots \to  E(R/\bm) \to 0.\]
 Our Theorem \ref{vanishing unramified} asserts that $H^n_I(R)=H^{n-1}_I(R)=0$. Hence, after applying the functor $\Gamma_I$ to the given injective resolution of $R$, one obtain the following exact sequence:
\[0\to H^{n-2}_I(R) \to \bigoplus_{I\subseteq \bp;\ \hgt(\bp)=n-2}E(R/\bp) \to \bigoplus_{I\subseteq \bq;\ \hgt(\bq)=n-1}E(R/\bq) \to E(R/\bm) \to 0,\]
which is clearly an injective resolution of $H^{n-2}_I(R)$. Therefore, we have $H^2_{\bm}H^{n-2}_I(R)=E(R/\bm)$ and $H^j_{\bm}H^{n-2}_I(R)=0$ for $j\neq 2$. This proves that case when $\dim(A)=2$.

Assume now $\dim(A)\geq 3$. We will pick an element $r\in \bm$ as follows (same as in the proof of \cite[Theorem 1.4]{ZhangHighestLyubeznikNumbers}). If $\Supp(H^{n-d+1}_I(R))\neq \{\bm\}$, then by prime avoidance we pick $r$ that is not in any minimal prime of $I$ nor in any minimal element of $\Supp(H^{n-d+1}_I(R))$ (which has finitely many associated primes as shown in \cite{LyubeznikUnramifiedRegular}). If $\Supp(H^{n-d+1}_I(R))= \{\bm\}$, then we pick $r\in \bm$ that is not in any minimal prime of $I$. Then $\dim(R/I+(r))=\dim(A)-1$ and $R/I+(r)$ is also equidimensional. Our theorem now follows from the following statements:
\begin{enumerate}
\item $H^d_{\bm}H^{n-d}_I(R)\cong H^{d-1}_{\bm}H^{n-d+1}_{I+(r)}(R)$, and
\item $\Gamma_{R/\sqrt{I+(r)}}$ is connected.
\end{enumerate}
These two statements appeared as Proposition 2.1 and Proposition 2.2, respectively, in \cite{ZhangHighestLyubeznikNumbers}. The proofs of these two statements in \cite{ZhangHighestLyubeznikNumbers} do {\it not} require the ring to contain a field. This completes the proof our theorem. 
\end{proof}

\begin{corollary}
If a $d$-dimensional noetherian local ring $A$ satisfies the Serre's $(S_2)$ condition, then $\lambda_{d,d}(A)=1$. 
\end{corollary}
\begin{proof}
When $A$ contains a field, this is known ({\it cf.} \cite[Theorem 4.6]{SurveyLNumbers}). Assume $A$ doesn't contain a field. According to our Theorem \ref{L numbers mixed char}, it suffices to show that the Hochster-Huneke graph $\Gamma_{\widetilde{A}}$ of $\widetilde{A}=\widehat{A^{sh}}$ is connected. Since $A$ is $S_2$, so is $\widetilde{A}$. Then \cite[Remark 2.4.1]{HartshorneCompleteIntersectionConnectedness} implies that $\widetilde{A}$ is equidimensional. Therefore, $\Gamma_{\widetilde{A}}$ must be connected by \cite[Theorem 3.6]{HochsterHunekeIndecomposable}.
\end{proof}

\begin{remark}
Let $(A,\bm,k)$ be a complete local ring of prime characteristic $p$. Then, by Cohen's Structure Theorem, one can write $A=R/I$ where $R=k[[x_1,\dots,x_n]]$ is a formal power series ring over $k$. Denote the maximal ideal of $R$ by $\bn$. One may consider 
\begin{equation}
\label{lambda char p}
\dim_k\Hom_R(k, H^i_{\bn}H^{n-j}_I(R))
\end{equation}
which agrees with the Lyubeznik number $\lambda_{i,j}(A)$ ({\it cf.} 
\cite[1.4,~4.1]{LyubeznikFinitenessLC}). On the other hand, $A$ can be also written as $R'/I'$ where $R'$ is a complete unramified regular local ring of mixed characteristic $(0,p)$. Denote the maximal ideal of $R'$ by $\bn'$ and $\dim(R')$ by $n'$. Following Definition \ref{definition of lambda}, one may consider 
\begin{equation}
\label{lambda mixed}
\dim_k\Hom_{R'}(k, H^i_{\bn'}H^{n'-j}_{I'}(R')).
\end{equation}

A natural question is that whether (\ref{lambda char p}) agrees with (\ref{lambda mixed}) (for fixed $i$ and $j$).

When $i=\dim(A)$ and $j=\dim(A)$, it follows immediately from our Theorem \ref{L numbers mixed char} and the main theorem in \cite{ZhangHighestLyubeznikNumbers} that the numbers (\ref{lambda mixed}) and (\ref{lambda char p}) coincide, both of which agree with the number of the connected components of the Hochster-Huneke graph of $A$. 
\end{remark}

\begin{remark}
The proof of Theorem \ref{L numbers mixed char} is an example of applying our Theorem \ref{vanishing unramified} to extend results, previously only known in equal-characteristic, to mixed characteristic. One may also apply Theorem \ref{vanishing unramified} to extend other results (for instance, some results in \cite{NunezBetancourtSpiroffWiit}) to mixed characteristic, which we will leave to another project.
\end{remark}

\section{Remarks and questions}
\label{observation and question}

In this section, we observe some natural mixed characteristic extensions of the links between vanishing of local cohomology and nilpotence of Frobenius action, and raise awareness of some open questions.

\subsection{Connections with Frobenius} Recall that Lyubeznik's vanishing theorem links the vanishing of $H^{n-i}_{\ba}(R)$ and the action of Frobenius on $H^{i}_{\bm}(R/\ba)$ where $(R,\bm)$ is a regular local ring of prime characteristic $p$ and $\ba$ is an ideal of $R$. In this section, we consider some (partial) extensions to mixed characteristic.

\begin{theorem}
\label{unramified extension}
Let $(R,\bm)$ be an $n$-dimensional unramified regular local ring of mixed characteristic $(0,p)$ and $\ba$ be an ideal of $R$. Assume that $p\in \ba$ (hence $R/\ba$ contains a field of characteristic $p$). If $H^i_{\bm}(R/\ba)$ is Frobenius nilpotent, then
\[H^{n-i}_{\ba}(R)=0.\]
\end{theorem} 
\begin{proof}
Set $\overline{R}=R/(p)$ and $\overline{\ba}=\ba/(p)$. Then $R/\ba=\overline{R}/\overline{\ba}$. Since $R$ is unramified, $\overline{R}$ is an $(n-1)$-dimensional regular local ring of characteristic $p$. By our assumption on $H^i_{\bm}(R/\ba)$ and Theorem \ref{theorem: vanishing of lc in char p}, we have $H^{(n-1)-i}_{\overline{\ba}}(\overline{R})=0$. The short exact sequence $0\to R\xrightarrow{\cdot p}R\to \overline{R}\to 0$ induces an exact sequence of local cohomology modules:
\[0=H^{(n-1)-i}_{\overline{\ba}}(\overline{R})=H^{(n-1)-i}_{\ba}(\overline{R})\to H^{n-i}_{\ba}(R)\xrightarrow{\cdot p}H^{n-i}_{\ba}(R).\]
Since $p\in \ba$, this forces $H^{n-i}_{\ba}(R)=0$.
\end{proof}

We ought to remark that Theorem \ref{unramified extension} only extends one implication in Theorem \ref{theorem: vanishing of lc in char p}. As to the other implication, we ask:
\begin{question}
\label{question: converse to extension to unramified}
Does the converse to Theorem \ref{unramified extension} hold?
\end{question}

When $i=1$, the answer is affirmative and it follows from the Second Vanishing Theorems in characteristic $p$ and in unramified mixed characteristic ({\it i.e.} our Theorem \ref{vanishing unramified}).

\begin{theorem}
Let $(R,\bm)$ be an $n$-dimensional unramified regular local ring of mixed characteristic $(0,p)$ and $\ba$ be an ideal of $R$. Assume that $p\in \ba$ (hence $R/\ba$ contains a field of characteristic $p$). Then $H^{n-1}_{\ba}(R)=0$ if and only if $H^1_{\bm}(R/\ba)$ is Frobenius nilpotent.
\end{theorem}
\begin{proof}
$\Leftarrow$ is a special case of Theorem \ref{unramified extension}; it remains to prove $\Rightarrow$. 

Set $\overline{R}=R/(p)$ and $\overline{\ba}=\ba/(p)$. Then $R/\ba=\overline{R}/\overline{\ba}$. Note that, since $R$ is regular, $H^n_{\ba}(R)=0$ by the Hartshorne-Lichtenbaum vanishing theorem. By Theorem \ref{vanishing unramified}, the punctured spectrum of $
\widetilde{R}\ba \widetilde{R}$ is connected ($\widetilde{R}$ is the completion of the strict henselization of the completion of $R$) and $\dim(R\ba)\geq 2$. Since $p\in \ba$, the same will hold for $\overline{R}$ and hence $H^{n-2}_{\overline{\ba}}(\overline{R})=H^{n-1}_{\overline{\ba}}(\overline{R})=0$ by the Second Vanishing Theorem in characteristic $p$. It now follows from Theorem \ref{theorem: vanishing of lc in char p} that $H^1_{\bm}(\overline{R}/\overline{\ba})=H^1_{\bm}(R/\ba)$ is Frobenius nilpotent.
\end{proof}

Question \ref{question: converse to extension to unramified} concerns a specific cohomological degree. It turns out that one can draw a weaker conclusion if one considers all vanishing above a specific cohomological degree.

\begin{proposition}
Let $(R,\bm)$ be an $n$-dimensional unramified regular local ring of mixed characteristic $(0,p)$ and $\ba$ be an ideal of $R$. Assume that $p\in \ba$ (hence $R/\ba$ contains a field of characteristic $p$). If $H^{n-j}_{\ba}(R)=0$ for all $j\leq t$ for a fixed integer $t$, then $H^{j-1}_{\bm}(R/\ba)$ is Frobenius nilpotent for all $j\leq t$.
\end{proposition}
\begin{proof}
Set $\overline{R}=R/(p)$ and $\overline{\ba}=\ba/(p)$. Then $R/\ba=\overline{R}/\overline{\ba}$. Consider the long exact sequence of local cohomology induced by $0\to R\xrightarrow{\cdot p}R\to \overline{R}\to 0$. Since $H^{n-j}_{\ba}(R)=0$ for all $j\leq t$, one has $H^{n-j}_{\overline{\ba}}(\overline{R})=H^{n-j}_{\ba}(\overline{R})=0$ for all $j\leq t$. Theorem \ref{theorem: vanishing of lc in char p} shows that 
\[H^{j-1}_{\bm}(R/\ba)=H^{j-1}_{\bm}(\overline{R}/\overline{\ba})=H^{(n-1)-(n-j)}_{\bm}(\overline{R}/\overline{\ba})\]
is Frobenius nilpotent for $j\leq t$.
\end{proof}

When $R$ is ramified, the situation seems to be much more mysterious since $R/(p)$ is no longer a regular ring. We are only able to obtain a weaker version of Theorem \ref{unramified extension} as follows.

\begin{theorem}
\label{extension to ramified}
Let $(R,\bm)$ be an $n$-dimensional {\bf ramified} regular local ring of mixed characteristic $(0,p)$ and $\ba$ be an ideal of $R$. Assume that $p\in \ba$ (hence $R/\ba$ contains a field of characteristic $p$). Assume that $H^j_{\bm}(R/\ba)$ is Frobenius nilpotent for $j\leq t$ for a fixed integer $t$, then 
\[H^{n+1-j}_{\ba}(R)=0\quad {\rm for\ }j\leq t.\]
\end{theorem}
\begin{proof}
Without loss of generality, we may assume that $R$ is complete. By Cohen's Structure Theorem, $R\cong V[[x_1,\dots,x_n]]/(p-f)$ where $f\in \bm^2$. Set $A=(V/pV)[[x_1,\dots,x_n]]$ (an $n$-dimensional regular local ring of characteristic $p$). We will denote the image of $f$ in $A$ by $f$ again. Then
\[R/(p)\cong A/(f).\] 
Set $\overline{R}=R/(p)$ and $\overline{\ba}=\ba/(p)$. We may view $\overline{\ba}$ as an ideal in $A/(f)$. Let $\bb$ be the ideal in $A$ such that $f\in \bb$ and $\bb/(f)=\overline{\ba}$. It is clear that
\[R/\ba\cong A/\bb,\]
and hence $H^j_{\bm}(A/\bb)$ is Frobenius nilpotent for $j\leq t$. Theorem \ref{theorem: vanishing of lc in char p} asserts that $H^{n-j}_{\bb}(A)=0$ for $j\leq t$. The exact sequence of local cohomology induced by $0\to A\xrightarrow{\cdot f}A\to A/(f)\to 0$ shows that
\[H^{n-j}_{\bb}(A/(f))=0\]
for $j\leq t$. (This is where we need to assume vanishing above a cohomological degree instead of vanishing at a single degree.) Consequently $H^{n-j}_{\ba}(R/(p))=H^{n-j}_{\overline{\ba}}(R/(p))=0$ for $j\leq t$.

Consider the exact sequence of local cohomology induced by the exact sequence $0\to R\xrightarrow{\cdot p}R\to \overline{R}\to 0$:
\[H^{n-j}_{\ba}(R/(p))\to H^{n+1-j}_{\ba}(R)\xrightarrow{\cdot p}H^{n+1-j}_{\ba}(R)\]
for $j\leq t$. Since $p\in \ba$, this forces $H^{n+1-j}_{\ba}(R)=0$ for $j\leq t$.
\end{proof}

Recall that the local cohomological dimension of an ideal $I$ in a noetherian ring $A$, denoted by $\lcd_A(I)$, is $\max\{j\mid H^j_I(A)\neq 0\}$.

For a noetherian local ring $(A,\bm)$ of prime characteristic $p$, its local cohomology modules $H^i_{\bm}(A)$ are equipped with an action of Frobenius $f:H^i_{\bm}(A)\to H^i_{\bm}(A)$ induced by the Frobenius endomorphism on $A$. The $F$-depth of $A$  is defined as 
\[\Fdepth(A):=\min\{j\mid H^j_{\bm}(A)\ {\rm is\ not\ nilpotent\ under\ }f\}.\]
The notion of $F$-depth of a local ring is introduced by Lyubeznik in \cite[Definition~4.1]{LyubeznikVanishingLCCharp} and is analogous to the notion of de Rham depth\footnote{An analogous notion for analytic subspaces of a complex manifold is introduced and investigated in \cite{RSW23}.} introduced by Ogus in \cite{OgusLocalCohomologicalDimension}.

We are in position to prove:
\begin{theorem}
\label{lcd in mixed char}
Let $(R,\bm)$ be an $n$-dimensional regular local ring of mixed characteristic $(0,p)$ and let $\ba$ be an ideal that contains $p$. Then
\begin{enumerate}
\item $\lcd_R(\ba)\leq n-\Fdepth(R/\ba)$, when $R$ is unramified; and
\item $\lcd_R(\ba)\leq n+1-\Fdepth(R/\ba)$, when $R$ is ramified.
\end{enumerate}
\end{theorem}

\begin{proof}
Since $R$ is regular and $p\in \ba$, it follows from Theorem \ref{unramified extension} (unramified case), for each $t\leq \Fdepth(R/\ba)$,  
\[H^i_{\ba}(R)=0\ {\rm for\ }i\geq n-t\]
Therefore,
\[\lcd_R(\ba)\leq n-\Fdepth(R/\ba).\]

The ramified case follows similarly from Theorem \ref{extension to ramified}.
\end{proof}

\subsection{Some Open questions} In \cite[Theorem~2.5]{HunekeLyubeznikVanishing}, Huneke and Lyubeznik proved an `induction theorem' which enables them to provide a unified proof of the Second Vanishing Theorem in equal-characteristic. Their induction theorem is a refinement of a theorem due to Faltings \cite[Satz~1]{FaltingsVanishing}. 

Before we can recall the theorem due to Huneke-Lyubeznik, we need to reproduce some definitions. For a local ring $A$, set
\[\mdim(A):=\min\{\dim(A/Q)|\ Q\ {\rm is\ a\ minimal\ prime\ of\ }A\},\]
and
\[c(I):=\embdim(A)=\mdim(A/I)\]
for every ideal $I$ of $A$, where $\embdim(A)$ denotes the embedding dimension of $A$.

\begin{theorem}[Faltings]
\label{Faltings vanishing}
Let $A$ be a complete local ring containing its separably closed residue field. Let $I$ be an ideal of $A$ and let $n>c(I)$ be an integer and $M$ be a finitely generated $A$-module. Assume that, for every integer $s$ with $0<s<c(I)$ and for every prime ideal $\fp\subset A$ with $\dim(A/\fp)>s$, $H^q_{IA_{\fp}}(M_{\fp})=0$ for all $q\geq n-s$. Then
\[H^q_I(M)=0,\quad \forall q\geq n.\]
\end{theorem}

\begin{theorem}[Huneke-Lyubeznik]
\label{HL vanishing}
Let $(A,\bm)$ be a complete noetherian local ring containing its separably closed residue field. Let $I$ be an ideal of $A$ and let $n>c(I)$ be an integer and $M$ be a finitely generated $A$-module. Assume that, for every integer $s$ with $0<s<c(I)$ and for all $q\geq n-s$, the following hold
\begin{enumerate}
\item $H^q_{IA_{\fp}}(M_{\fp})=0$ for all $\fp\in\Spec(A)$ such that $I\subseteq \fp$ and $\dim(A/\fp)>s+1$
\item $H^q_{IA_{\fp}}(M_{\fp})=0$ for all $\fp\in\Spec(A)$ such that $I\subseteq \fp$, $\dim(A/\fp)=s+1$, and $\fp+\fq$ is $\bm$-primary for some minimal prime $\fq$ of $I$.
\end{enumerate}
Then
\[H^q_I(M)=0,\quad \forall q\geq n.\]
\end{theorem}

These two theorems have produced a family of results on local cohomological dimension, in equal-characteristic. More specifically, Theorem \ref{HL vanishing} implies that the Second Vanishing Theorem holds for all regular local rings of equal-characteristic.

\begin{question}
\label{extending induction thm to mixed}
Do Theorems \ref{Faltings vanishing} and \ref{HL vanishing} hold in mixed characteristic?
\end{question}

A positive answer to Question \ref{extending induction thm to mixed} will produce a new family of results on local cohomological dimension in mixed characteristic. More specifically, extending Theorem \ref{HL vanishing} to local rings of mixed characteristic is another approach to proving the Second Vanishing Theorem in mixed characteristic. To the best of our knowledge, Question \ref{extending induction thm to mixed} is wide open.

One may also consider extending the Second Vanishing Theorem to non-regular rings. In \cite{LyubeznikSurvey}, Lyubeznik asked the following.
\begin{question}[Lyubeznik]
\label{Lyubeznik question}
Let $(A,\bm)$ be a complete local domain of dimension $d$ whose residue field is separably closed. 
\begin{enumerate}
\item Find necessary and sufficient condition on $I$ under which $H^j_I(A)=0$ for all $j>d-2$. 
\item Let $I$ be a prime ideal. Assume that $I+\fp$ is not $\bm$-primary for every height-1 prime ideal $\fp$. Is it true that $H^j_I(A)=0$ for all $j>d-2$?
\end{enumerate}
\end{question}

As shown in \cite[7.7]{HochsterZhangContent}, Question \ref{Lyubeznik question}(b) has a negative answer even for complete intersections as stated. Hence additional assumptions may be needed. More specifically, we ask

\begin{question}
\begin{enumerate}
\item Let $R$ be an equal-characteristic regular local ring whose residue field is separably closed and let $G$ be a linearly reductive group acting on $R$. Does Question \ref{Lyubeznik question} have a positive answer for the invariant subring $R^G$?
\item Analogously, assume $R$ is a polynomial ring over a separably closed field and let $G$ be a linearly reductive group acting on $R$. Does Question \ref{Lyubeznik question} have a positive answer for homogeneous ideals in $R^G$?
\end{enumerate}
\end{question}

\subsection*{Acknowledgement} 

The author is grateful to the anonymous referees and Gennady Lyubeznik for their suggestions which greatly improve the exposition of this paper. The author also would like to thank Bhargav Bhatt and Gennady Lyubeznik for related conversations and Luis N\'{u}\~{n}ez-Betancourt for comments on a draft. 

\bibliographystyle{skalpha}
\bibliography{CommonBib}

\def\cprime{$'$} \def\cprime{$'$}
  \def\cfudot#1{\ifmmode\setbox7\hbox{$\accent"5E#1$}\else
  \setbox7\hbox{\accent"5E#1}\penalty 10000\relax\fi\raise 1\ht7
  \hbox{\raise.1ex\hbox to 1\wd7{\hss.\hss}}\penalty 10000 \hskip-1\wd7\penalty
  10000\box7}
\providecommand{\bysame}{\leavevmode\hbox to3em{\hrulefill}\thinspace}
\providecommand{\MR}{\relax\ifhmode\unskip\space\fi MR}
\providecommand{\MRhref}[2]{%
  \href{http://www.ams.org/mathscinet-getitem?mr=#1}{#2}
}
\providecommand{\href}[2]{#2}
\begin{thebibliography}{HNnBPW18}

\bibitem[BB05]{BlickleBondu}
{\sc M.~Blickle and R.~Bondu}: \emph{Local cohomology multiplicities in terms
  of \'etale cohomology}, Ann. Inst. Fourier (Grenoble) \textbf{55} (2005),
  no.~7, 2239--2256. {\sf\scriptsize 2207383 (2007d:14009)}

\bibitem[Bol79]{BollobasBookGraphTheory}
{\sc B.~Bollob\'{a}s}: \emph{Graph theory}, Graduate Texts in Mathematics,
  vol.~63, Springer-Verlag, New York-Berlin, 1979, An introductory course.
  {\sf\scriptsize 536131}

\bibitem[DSZ23]{DattaSwitalaZhang}
{\sc R.~Datta, N.~Switala, and W.~Zhang}: \emph{Annihilators of {$D$}-modules
  in mixed characteristic}, Math. Res. Lett. \textbf{30} (2023), no.~3,
  721--732. {\sf\scriptsize 4696428}

\bibitem[Fal80]{FaltingsVanishing}
{\sc G.~Faltings}: \emph{\"{U}ber lokale {K}ohomologiegruppen hoher {O}rdnung},
  J. Reine Angew. Math. \textbf{313} (1980), 43--51. {\sf\scriptsize 552461}

\bibitem[GLS98]{GarciaSabbah}
{\sc R.~Garc{\'{\i}}a~L{{\'o}}pez and C.~Sabbah}: \emph{Topological computation
  of local cohomology multiplicities}, Collect. Math. \textbf{49} (1998),
  no.~2-3, 317--324, Dedicated to the memory of Fernando Serrano.
  {\sf\scriptsize 1677136 (2000a:13029)}

\bibitem[Har62]{HartshorneCompleteIntersectionConnectedness}
{\sc R.~Hartshorne}: \emph{Complete intersections and connectedness}, Amer. J.
  Math. \textbf{84} (1962), 497--508. {\sf\scriptsize 0142547 (26 \#116)}

\bibitem[Har67]{HartshorneLocalCohomology}
{\sc R.~Hartshorne}: \emph{Local cohomology}, A seminar given by A.
  Grothendieck, Harvard University, Fall, vol. 1961, Springer-Verlag, Berlin,
  1967. {\sf\scriptsize MR0224620 (37 \#219)}

\bibitem[Har68]{HartshorneCohomologicalDimension}
{\sc R.~Hartshorne}: \emph{Cohomological dimension of algebraic varieties},
  Ann. of Math. (2) \textbf{88} (1968), 403--450. {\sf\scriptsize 0232780 (38
  \#1103)}

\bibitem[HS77]{HartshorneSpeiserLocalCohomologyInCharacteristicP}
{\sc R.~Hartshorne and R.~Speiser}: \emph{Local cohomological dimension in
  characteristic {$p$}}, Ann. of Math. (2) \textbf{105} (1977), no.~1, 45--79.
  {\sf\scriptsize MR0441962 (56 \#353)}

\bibitem[HNnBPW18]{HNBPW_JA_2018}
{\sc D.~J. Hern\'{a}ndez, L.~N\'{u}\~{n}ez Betancourt, F.~P\'{e}rez, and E.~E.
  Witt}: \emph{Cohomological dimension, {L}yubeznik numbers, and connectedness
  in mixed characteristic}, J. Algebra \textbf{514} (2018), 442--467.
  {\sf\scriptsize 3853071}

\bibitem[HH94]{HochsterHunekeIndecomposable}
{\sc M.~Hochster and C.~Huneke}: \emph{Indecomposable canonical modules and
  connectedness}, Commutative algebra: syzygies, multiplicities, and birational
  algebra (South Hadley, MA, 1992), Contemp. Math., vol. 159, Amer. Math. Soc.,
  Providence, RI, 1994, pp.~197--208. {\sf\scriptsize MR1266184 (95e:13014)}

\bibitem[HZ18]{HochsterZhangContent}
{\sc M.~Hochster and W.~Zhang}: \emph{Content of local cohomology, parameter
  ideals, and robust algebras}, Trans. Amer. Math. Soc. \textbf{370} (2018),
  no.~11, 7789--7814. {\sf\scriptsize 3852449}

\bibitem[HL90]{HunekeLyubeznikVanishing}
{\sc C.~Huneke and G.~Lyubeznik}: \emph{On the vanishing of local cohomology
  modules}, Invent. Math. \textbf{102} (1990), no.~1, 73--93. {\sf\scriptsize
  1069240 (91i:13020)}

\bibitem[Hun92]{HunekeProblemsLC}
{\sc C.~Huneke}: \emph{Problems on local cohomology}, Free resolutions in
  commutative algebra and algebraic geometry ({S}undance, {UT}, 1990), Res.
  Notes Math., vol.~2, Jones and Bartlett, Boston, MA, 1992, pp.~93--108.
  {\sf\scriptsize 1165320}

\bibitem[HK91]{HunekeKohCofiniteness}
{\sc C.~Huneke and J.~Koh}: \emph{Cofiniteness and vanishing of local
  cohomology modules}, Math. Proc. Cambridge Philos. Soc. \textbf{110} (1991),
  no.~3, 421--429. {\sf\scriptsize 1120477 (92g:13021)}

\bibitem[Lyu93]{LyubeznikFinitenessLC}
{\sc G.~Lyubeznik}: \emph{Finiteness properties of local cohomology modules (an
  application of {$D$}-modules to commutative algebra)}, Invent. Math.
  \textbf{113} (1993), no.~1, 41--55. {\sf\scriptsize 1223223 (94e:13032)}

\bibitem[Lyu00]{LyubeznikUnramifiedRegular}
{\sc G.~Lyubeznik}: \emph{Finiteness properties of local cohomology modules for
  regular local rings of mixed characteristic: the unramified case}, Comm.
  Algebra \textbf{28} (2000), no.~12, 5867--5882, Special issue in honor of
  Robin Hartshorne. {\sf\scriptsize 1808608 (2002b:13028)}

\bibitem[Lyu02]{LyubeznikSurvey}
{\sc G.~Lyubeznik}: \emph{A partial survey of local cohomology}, Local
  cohomology and its applications ({G}uanajuato, 1999), Lecture Notes in Pure
  and Appl. Math., vol. 226, Dekker, New York, 2002, pp.~121--154.
  {\sf\scriptsize 1888197 (2003b:14006)}

\bibitem[Lyu06a]{LyubeznikSomeLCInvariants}
{\sc G.~Lyubeznik}: \emph{On some local cohomology invariants of local rings},
  Math. Z. \textbf{254} (2006), no.~3, 627--640. {\sf\scriptsize 2244370}

\bibitem[Lyu06b]{LyubeznikVanishingLCCharp}
{\sc G.~Lyubeznik}: \emph{On the vanishing of local cohomology in
  characteristic {$p>0$}}, Compos. Math. \textbf{142} (2006), no.~1, 207--221.
  {\sf\scriptsize 2197409 (2007b:13029)}

\bibitem[NnB13]{NunezBetancourtIJM2013}
{\sc L.~N\'{u}\~{n}ez Betancourt}: \emph{Local cohomology modules of polynomial
  or power series rings over rings of small dimension}, Illinois J. Math.
  \textbf{57} (2013), no.~1, 279--294. {\sf\scriptsize 3224571}

\bibitem[NnBSW19]{NunezBetancourtSpiroffWiit}
{\sc L.~N\'{u}\~{n}ez Betancourt, S.~Spiroff, and E.~E. Witt}:
  \emph{Connectedness and {L}yubeznik numbers}, Int. Math. Res. Not. IMRN
  (2019), no.~13, 4233--4259. {\sf\scriptsize 3978438}

\bibitem[NnBW13]{NunezWittLNMixedChar}
{\sc L.~N\'u\~nez Betancourt and E.~E. Witt}: \emph{Lyubeznik numbers in mixed
  characteristic}, Math. Res. Lett. \textbf{20} (2013), no.~6, 1125--1143.
  {\sf\scriptsize 3228626}

\bibitem[NnBW14]{NunezWittGenLN}
{\sc L.~N\'u\~nez Betancourt and E.~E. Witt}: \emph{Generalized {L}yubeznik
  numbers}, Nagoya Math. J. \textbf{215} (2014), 169--201. {\sf\scriptsize
  3296589}

\bibitem[NnBWZ16]{SurveyLNumbers}
{\sc L.~N\'{u}\~{n}ez Betancourt, E.~E. Witt, and W.~Zhang}: \emph{A survey on
  the {L}yubeznik numbers}, Mexican mathematicians abroad: recent
  contributions, Contemp. Math., vol. 657, Amer. Math. Soc., Providence, RI,
  2016, pp.~137--163. {\sf\scriptsize 3466449}

\bibitem[Ogo91]{OgomaDirectLimit}
{\sc T.~Ogoma}: \emph{Noetherian property of inductive limits of {N}oetherian
  local rings}, Proc. Japan Acad. Ser. A Math. Sci. \textbf{67} (1991), no.~3,
  68--69. {\sf\scriptsize 1105524}

\bibitem[Ogu73]{OgusLocalCohomologicalDimension}
{\sc A.~Ogus}: \emph{Local cohomological dimension of algebraic varieties},
  Ann. of Math. (2) \textbf{98} (1973), 327--365. {\sf\scriptsize 0506248 (58
  \#22059)}

\bibitem[PS73]{PeskineSzpiroDimensionProjective}
{\sc C.~Peskine and L.~Szpiro}: \emph{Dimension projective finie et cohomologie
  locale. {A}pplications \`a la d\'emonstration de conjectures de {M}.
  {A}uslander, {H}. {B}ass et {A}. {G}rothendieck}, Inst. Hautes \'Etudes Sci.
  Publ. Math. (1973), no.~42, 47--119. {\sf\scriptsize MR0374130 (51 \#10330)}

\bibitem[RSW21]{ReicheltSaitoWalther}
{\sc T.~Reichelt, M.~Saito, and U.~Walther}: \emph{Dependence of {L}yubeznik
  numbers of cones of projective schemes on projective embeddings}, Selecta
  Math. (N.S.) \textbf{27} (2021), no.~1, Paper No. 6, 22. {\sf\scriptsize
  4202748}

\bibitem[RSW23]{RSW23}
{\sc T.~Reichelt, M.~Saito, and U.~Walther}: \emph{Topological calculation of
  local cohomological dimension}, J. Singul. \textbf{26} (2023), 13--22.
  {\sf\scriptsize 4613755}

\bibitem[SW11]{SinghWaltherBockstein}
{\sc A.~K. Singh and U.~Walther}: \emph{Bockstein homomorphisms in local
  cohomology}, J. Reine Angew. Math. \textbf{655} (2011), 147--164.
  {\sf\scriptsize 2806109}

\bibitem[{Sta}]{stacks-project}
{\sc T.~{Stacks project authors}}: \emph{The stacks project}.

\bibitem[Swi15]{SwitalaBLMS}
{\sc N.~Switala}: \emph{Lyubeznik numbers for nonsingular projective
  varieties}, Bull. Lond. Math. Soc. \textbf{47} (2015), no.~1, 1--6.
  {\sf\scriptsize 3312957}

\bibitem[Wal01]{WaltherLyubeznikNumbers}
{\sc U.~Walther}: \emph{On the {L}yubeznik numbers of a local ring}, Proc.
  Amer. Math. Soc. \textbf{129} (2001), no.~6, 1631--1634 (electronic).
  {\sf\scriptsize 1814090 (2001m:13026)}

\bibitem[Zha07]{ZhangHighestLyubeznikNumbers}
{\sc W.~Zhang}: \emph{On the highest {L}yubeznik number of a local ring},
  Compos. Math. \textbf{143} (2007), no.~1, 82--88. {\sf\scriptsize 2295196
  (2008g:13029)}

\bibitem[Zha11]{ZhangLyubeznikNumbersProjectiveSchemes}
{\sc W.~Zhang}: \emph{Lyubeznik numbers of projective schemes}, Adv. Math.
  \textbf{228} (2011), no.~1, 575--616. {\sf\scriptsize 2822240 (2012j:13027)}

\end{thebibliography}
\end{document}